\documentclass[10pt]{amsart}

\let\oldmarginpar\marginpar
\renewcommand\marginpar[1]{\-\oldmarginpar[\raggedleft\footnotesize #1]
{\raggedright\footnotesize #1}}

\DeclareMathSizes{6}{18}{12}{8}

\usepackage{mathptmx}
\usepackage{amsmath}
\usepackage{amscd}
\usepackage{amssymb}
\usepackage{amsthm}
\usepackage{xspace}
\usepackage[all,tips]{xy}
\usepackage[dvips]{graphicx}
\usepackage{verbatim}
\usepackage{syntonly}
\usepackage{hyperref}
\usepackage{color}
\usepackage{fancyhdr}
\usepackage{mathdots}
\usepackage{ifthen}
\usepackage{enumitem}

\theoremstyle{plain}
\newtheorem{thm}{Theorem}[section]
\newtheorem{cor}[thm]{Corollary}
\newtheorem{prop}[thm]{Proposition}
\newtheorem{lemma}[thm]{Lemma}

\newtheorem{sch}[thm]{Scholium}
\newtheorem{claim}[thm]{Claim}

\newcommand{\showcomments}{yes}
\renewcommand{\showcomments}{no}

\newsavebox{\commentbox}
{\ifthenelse{\equal{\showcomments}{yes}}%
{\footnotemark
        \begin{lrbox}{\commentbox}
        \begin{minipage}[t]{1.25in}\raggedright\sffamily\tiny
        \footnotemark[\arabic{footnote}]}
{\begin{lrbox}{\commentbox}}}%
{\ifthenelse{\equal{\showcomments}{yes}}%
{\end{minipage}\end{lrbox}\marginpar{\usebox{\commentbox}}}
{\end{lrbox}}}

\DeclareMathOperator{\SL}{SL} 
\DeclareMathOperator{\GL}{GL} 
 
 \DeclareMathOperator{\SO}{SO}

\DeclareMathOperator{\SU}{SU} 
\DeclareMathOperator{\Sp}{Sp} 
 
 \DeclareMathOperator{\Comm}{Comm}

\DeclareMathOperator{\lcm}{lcm}

\DeclareMathOperator{\Res}{Res}

\DeclareMathOperator{\leng}{Length}

\DeclareMathOperator{\diag}{diag}
\DeclareMathOperator{\id}{Id}\DeclareMathOperator{\rk}{rank}

\newcommand{\rotsubset}{\mathbin{\rotatebox[origin=c]{90}{$\subset$}}}

\newcommand{\Fr}[1]{\ensuremath{\mathfrak{#1}}}

\newcommand{\Hy}{\ensuremath{\mathbb{H}}}
\newcommand{\N}{\ensuremath{\mathbb{N}}}
\newcommand{\Q}{\ensuremath{\mathbb{Q}}}
\newcommand{\R}{\ensuremath{\mathbb{R}}}
\newcommand{\Z}{\ensuremath{\mathbb{Z}}}
\newcommand{\C}{\ensuremath{\mathbb{C}}}

\usepackage{fancyhdr}

\fancypagestyle{plain}

\title{\textbf{Arithmetic Progressions in the Primitive Length Spectrum}}

\begin{document}

\title{Arithmetic Progressions in the Primitive Length Spectrum}
\author{Nicholas Miller}
\address{Department of Mathematics\\Purdue University\\West Lafayette, IN 47907}
\email{mille965@purdue.edu}


\maketitle

\begin{abstract}
In this article, we prove that every arithmetic locally symmetric orbifold of classical type without Euclidean or compact factors has arbitrarily long arithmetic progressions in its primitive length spectrum.
Moreover, we show the stronger property that every primitive length occurs in arbitrarily long arithmetic progressions in its primitive length spectrum.
This confirms one direction of a conjecture of Lafont--McReynolds, which states that the property of having every primitive length occur in arbitrarily long arithmetic progressions characterizes the arithmeticity of such spaces.
\end{abstract}


\section{Introduction.}

\noindent Given a smooth manifold $M$ and a Riemannian metric $g$ on $M$, the \textbf{(primitive) length spectrum} is the collection of all lengths of (primitive) closed geodesics counted with multiplicity.
When $M$ is closed and $g$ is negatively curved, it is well known that there is a strong connection between the eigenvalue spectrum of the Laplacian acting on $L^2(M)$ and the length spectrum \cite{DG} \cite{Gangolli} \cite{PrasadRap}.
These spectra are known to determine basic geometric invariants like the dimension, volume, and total scalar curvature of $(M,g)$. 

Over the past several decades, there has been a great deal of activity in understanding to what extent either the eigenvalue or geodesic length spectrum determines the metric $g$.
For instance when $M$ is an arithmetic hyperbolic surface, Reid \cite{Reid} proved that the length spectrum determines its commensurability class.
Moreover Chinburg--Hamilton--Long--Reid \cite{CHLR} have shown that the length spectrum of $M$ determines its commensurability class amongst arithmetic, hyperbolic $3$--manifolds.
More recently, Prasad and Rapinchuk \cite{PrasadRap} have given a complete characterization of when the length spectrum determines the commensurability class of $M$ for all non-positively curved, arithmetic manifolds, including families of counterexamples starting in dimension $5$.
Their work uses the fact that, in the arithmetic setting, there are deep connections between the study of closed geodesics on  non-positively curved, arithmetic manifolds and the number theory of the corresponding arithmetic subgroup.
This number theoretic connection is a theme that is abundant in the study of the length spectrum of arithmetic manifolds.

Along these lines, several results exhibit a striking analogy between features of primitive geodesics in the length spectrum and features of prime numbers in the integers, or more generally in the rings of integers of number fields.
For instance if $M$ is compact, results of Margulis \cite{Margulis} show that the asymptotic growth rate of primitive, closed geodesics on $M$ is analogous to that of the prime number theorem.
If one further assumes the manifold $M$ is arithmetic, there are a number of other results further elucidating the analogy between primitive geodesics and primes.
For instance, Selberg \cite{Selberg} noticed that it is possible to form a zeta function, defined by a certain product over primitive geodesics, which can be shown to enjoy analogous properties to the Riemann zeta function including the anaologue of the Riemann hypothesis.
More examples of such phenomena are the holonomy distribution theorems of Parry--Pollicott \cite{ParPol}, Sarnak--Wakayama \cite{SarWaka}, and Margulis--Mohammadi--Oh \cite{MMO}.
Such theorems are meant to give the analogue of the Chebotarev density theorem for closed geodesics on an arithmetic, hyperbolic manifold.
The purpose of this article is to expand the collection of such theorems, by providing the geodesic analogue of the Green--Tao theorem on progressions in the prime numbers \cite{GT}.

Recently, Lafont--McReynolds \cite{LaMcR} have proved such an analogue for primitive, closed geodesics on the modular curve.
Indeed, Lafont--McReynolds show that the primitive length spectrum of the modular curve contains arbitrarily long arithmetic progressions.
Moreover, they show that a much stronger property holds, namely that every primitive length occurs in a $k$-term arithmetic progression for any $k$.
It is then conjectured that having the stronger property, i.e. that every primitive length occurs in arbitrarily long arithmetic progressions, is indeed a characterization of the arithmeticity of $M$.
The goal of this article is to prove one direction of this conjecture by generalizing the result of Lafont--McReynolds to all arithmetic locally symmetric orbifolds of classical type.
In particular, we prove the following theorems.

\begin{thm}\label{ProgTheorem}
If $M$ is an arithmetic locally symmetric orbifold of classical type, then the primitive length spectrum $\mathcal{L}_p(M)$ has arbitrarily long arithmetic progressions.
More precisely, for every $k$ there is a $k$-term arithmetic progression of the form $\{(ai+b)\ell\}_{i=1}^k$ in $\mathcal{L}_p(M)$, where $\ell$ is the length of some primitive closed geodesic and $\{ai+b\}_{i=1}^k$ is a subset of the natural numbers.
\end{thm}

\noindent Moreover, we show the strong version of this theorem for such spaces.

\begin{thm}\label{EveryPrimTheorem}
If $M$ is an arithmetic locally symmetric orbifold of classical type, then every primitive length of $\mathcal{L}_p(M)$ occurs in arbitrarily long arithmetic progressions.
More precisely, for any $\ell\in\mathcal{L}_p(M)$ and any $k$, there is a $k$-term arithmetic progression of the form $\{(ai+b)\ell\}_{i=1}^k$ in $\mathcal{L}_p(M)$, where $\{ai+b\}_{i=1}^k\subset\N$.
\end{thm}

\noindent Recall that two manifolds $M$ and $M'$ are said to be \textbf{commensurable} if there exists some $X$ which is a finite sheeted cover of both $M$ and $M'$.
In \cite{LaMcR} it is shown that any orbifold commensurable to the modular curve also carries arbitrarily long arithmetic progressions in the primitive length spectrum.
Using their ideas, we prove this in the general case.

\begin{thm}\label{CommInvTheorem}
The property of having arbitrarily long arithmetic progressions in the primitive length spectrum and the property of having every primitive length occur in arbitrarily long arithmetic progressions in the primitive length spectrum are commensurability invariants.
That is to say, if $M$ is an arithmetic locally symmetric orbifold of classical type commensurable with $M'$, then $\mathcal{L}_p(M)$ has arbitrarily long arithmetic progressions if and only if $\mathcal{L}_p(M')$ has arbitrarily long arithmetic progressions.
Moreover every primitive length occurs in an arbitrarily long arithmetic progressions in $\mathcal{L}_p(M)$ if and only if the same property holds for $\mathcal{L}_p(M')$.
\end{thm}

\noindent As a corollary of these theorems, we see that any orbifold which contains a totally geodesic embedding of an orbifold commensurable to an arithmetic locally symmetric space also has arbitrarily long arithmetic progressions in the primitive length spectrum.

\paragraph{\textbf{Acknowledgements.}}
The author would first and foremost like to give special thanks to his advisor, Ben McReynolds, for numerous hours of conversation throughout the various stages of completion of this project.
The author would additionally like to thank Britain Cox, Jean-Fran\c cois Lafont, and Abhishek Parab for helpful and stimulating conversations during the genesis of this project.


\section{Notation and Preliminaries}

\subsection{Notation}\label{NotationSection}

Throughout the entire text, $G$ will always denote a fixed connected, absolutely simple, adjoint algebraic group of classical type defined over a number field $K$.
For any finite place $v$ of $K$, we will use $G_v$ to denote the $K_v$ points of $G$, $f_v$ to denote the residue field of $K_v$, and $q_v$ to denote the cardinality of $f_v$.
We will constantly conflate the notion of $v$ as an element of $\mathcal{O}_K$, a prime ideal of $\mathcal{O}_K$, and as a valuation, however in what follows this will not cause any issue.
We may use Weil restriction to view $G$ as a $\Q$-algebraic group, from which we have the isomorphism $G(K)\cong \Res_{K/\Q}(G)(\Q)$.
Through this, we will use $\mathcal{G}$ to denote the $\R$ points of this restriction, $\Res_{K/\Q}(G)(\R)$, regarded as a real Lie group.
We will always assume that $\mathcal{G}$ has no compact or Euclidean factors and that $\Gamma$ is a lattice in $\mathcal{G}$.
Our lattices will always be irreducible.
For any $\Gamma<\mathcal{G}$ we define the commensurator as
$$\Comm(\Gamma)=\{g\in\mathcal{G}\mid \Gamma\cap g\Gamma g^{-1}<_{\text{f.i.}}\Gamma,~g\Gamma g^{-1}\},$$
where $<_{\text{f.i}}$ means is a subgroup of finite index.
We now list all simple real forms of classical type up to isogeny and their explicit matrix realizations, which we use in the sequel (\cite{PRBook} or \cite{Witte-Morris} for instance).
Let $M$ be the square matrix with $1$s on the anti-diagonal and $0$s everywhere else, namely
$$M=\begin{pmatrix}
&&1\\
&\iddots&\\
1&&
\end{pmatrix},$$
and let 
$$T_{p,q}=\begin{pmatrix}
\id_{q-p,0}&0\\
0&M
\end{pmatrix},$$
where $\id_{p,q}=\diag(1,\dots,1,-1,\dots,-1)$.
We then realize the $\R$-forms of $\mathcal{G}$ as follows.

\begin{itemize}[leftmargin=5.5mm]
\item[]\textbf{Type $^1A$.}
Then $\mathcal{G}$ is $\SL(m,\R)$, $\SL(m,\Hy)$, or $\SL(m,\C)$ considered as a real Lie group, where $\Hy$ are Hamilton's quaternions.
\item[]
\item[]\textbf{Type $^2A$.}
Then $\mathcal{G}=\SU(p,q)$ with $1\le p\le q$, which we write as $\{A\in\SL(n,\C)\mid A^*T_{p,q}A=T_{p,q}\}$, where $A^*$ is the conjugate transpose of $A$.
\item[]
\item[]\textbf{Type $B$ or $^1D$.}
Then $\mathcal{G}=\SO(p,q)$ or $\SO(m,\C)$ considered as a real Lie group, where $1\le p\le q$.
We write $\SO(p,q)=\{S\in\SL(n,\R)\mid A^TT_{p,q}A=T_{p,q}\}$.
\item[]
\item[]\textbf{Type $C$.}
Then $\mathcal{G}=\Sp(p,q)$, with $1\le p\le q$.
Hence $\mathcal{G}=\{S\in\SL(m,\R)\mid A^TR_{p,q}A=R_{p,q}\}$, where
$$R_{p,q}=
\begin{pmatrix}
0&I_{p,q}\\
-I_{p,q}&0
\end{pmatrix}.$$
\item[]
\item[]\textbf{Type $^2D$.}
Then $\mathcal{G}=\SO(m,\Hy)$, i.e. matrices with coefficients in $\Hy$ that preserve the non-degenerate symmetric form $M$.
\end{itemize}

\noindent Each of these groups embeds inside $\SL(n,\R)$ or $\SL(n,\C)$ for some $n$, which we fix now and through the rest of the paper.
When we refer to a root $\alpha$ or $\mathcal{G}$, we mean with respect to the decomposition afforded by the standard choice of maximal torus.
When we discuss ``the $\mathcal{O}_K$ points of $G$'' in what follows, we will mean the matrices with $\mathcal{O}_K$ entries under some fixed matrix realization of $G$ inside $\SL_n$ such that $\mathcal{G}$ splits as a product of simple factors with the above realizations.
By the definition of arithmeticity, our results will be independent of this choice of realization.


\subsection{Locally symmetric spaces and their geodesics}
Let $\mathcal{K}$ denote the maximal compact subgroup of $\mathcal{G}$, then the corresponding quotient $\mathcal{K}\!\setminus\!\mathcal{G}$ is a symmetric space.
Quotienting this space by a discrete subgroup $\Gamma$ of $\mathcal{G}$ gives a locally symmetric orbifold $M=\mathcal{K}\!\setminus\!\mathcal{G}/\Gamma$, which is a manifold when $\Gamma$ is torsion free.
When $\mathcal{G}$ is a classical group we say that the associated locally symmetric orbifold $\mathcal{K}\!\setminus\!\mathcal{G}/\Gamma$ is of \textbf{classical type} and when $\Gamma$ is an arithmetric lattice, we say that $\mathcal{K}\!\setminus\!\mathcal{G}/\Gamma$ is \textbf{arithmetic}.
We wish to study closed geodesics on such spaces, hence we briefly recall some facts about the dictionary between semisimple elements in the lattice $\Gamma$ and closed geodesics on $M$.
All of the relevant details can be found in \cite[Section 8]{PrasadRap} and \cite{Witte-Morris}.

Recall an element $\gamma\in\Gamma$ is semisimple if it is diagonalizable over $\C$.
Given a semisimple element $\gamma$, let $T$ be the maximal $\R$-torus containing $\gamma$.
Then we write $T$ as a direct product of $T_c$ and $T_i$ where $T_c$ is a maximal compact subgroup of $T$ and $T_i$ is a maximal $\R$-split torus in $T$.
We may then decompose $\gamma$ into elliptic and hyperbolic parts, i.e. $\gamma=\gamma_e\cdot\gamma_h$ where $\gamma_e\in T_c$ and $\gamma_i\in T_i$.
We call a semisimple element $\gamma=\gamma_{e}\cdot\gamma_h$ \textbf{hyperbolic} if $\gamma_h\neq1$.
We will use the notation $\Gamma^{hyp}$ to denote the set of all semisimple elements which are hyperbolic.
With this in mind, let $z\in\mathcal{G}$ be any element such that $zTz^{-1}$ is invariant under the Cartan involution associated to the decomposition of the Lie algebra, $\Fr{g}=\Fr{t}\oplus\Fr{p}$, afforded by $T$.
Then there is some $X\in\Fr{p}$ such that $\gamma_h=z^{-1}\exp(X)z$ and consequently
\begin{align*}
\varphi:\R&\to M\\
t&\mapsto \mathcal{K}\exp(tX)z,
\end{align*}
parametrizes a closed geodesic $c_\gamma$ on $M$ passing through the point $\mathcal{K}z$.
Thus every conjugacy class of hyperbolic element gives rise to a closed geodesic on $M$, which we denote $c_\gamma$ and refer to as the associated geodesic to $\gamma$.
Prasad--Rapinchuk have shown that this in fact characterizes all closed geodesics on $M$.

\begin{prop}[\cite{PrasadRap}, Proposition 8.2]
Every closed geodesic on $M$ is of the form $c_\gamma$ for some $\gamma\in\Gamma^{hyp}$.
Moreover, the length of $c_\gamma$ is given by
$$\leng(c_\gamma)=\frac{1}{n_\gamma}\sqrt{\sum_{\alpha\in\Phi(\mathcal{G},T)}(\log|\alpha(\gamma)|)^2},$$
where $\Phi(\mathcal{G},T)$ are the roots of $\mathcal{G}$ with respect to the torus $T$ and $n_\gamma$ is some natural number.
\end{prop}

\begin{cor}
There is a one to one correspondence between (primitive) closed geodesics on $M$ and (primitive) conjugacy classes $[\gamma]_\Gamma$, where $\gamma\in\Gamma^{hyp}$.
Additionally for any hyperbolic $\gamma$, $\leng(c_{\gamma^n})=n\leng(c_\gamma)$.
\end{cor}

\noindent Hence the study of primitive geodesics and their lengths can be reduced to the study of the associated conjugacy classes of primitive hyperbolic elements in $\Gamma$, a dictionary we will frequently exploit.


\subsection{Bruhat--Tits theory}
We also require some results from Bruhat--Tits theory, all of the details of which can be found in \cite{BT1}, \cite{BT2}, and \cite{Tits}.
Given $G_v$ over the local field $K_v$ and a fixed maximal $K_v$-split torus $T_v$, we can form the associated apartment $\mathcal{A}_v=\mathcal{A}(G_v,K_v,T_v)$, defined as the affine space under $X_*(T_v)\otimes\R$, where $X_*(T_v)$ denotes the cocharacters of $T_v$.
This apartment carries an action by the normalizer of the torus, $N_{G_v}(T_v)$.
To upgrade this action to the full group $G_v$, we define the building $\mathcal{B}(G_v,K_v)=G_v\times\mathcal{A}_v/\!\sim$, where $\sim$ is a suitable equivalence relation so that the left action of $G_v$ is compatible with the action of $N_{G_v}(T_v)$ on $\mathcal{A}_v$.
Throughout the text we will abusively use the shorthand $\mathcal{B}_v$ when possible.

The building $\mathcal{B}_v$ carries a chamber structure, dictated by the vanishing sets of affine linear functionals.
Given a facet $F$ in $\mathcal{B}_v$, the stabilizer of $F$ in $G_v$ gives rise to a \textbf{parahoric} subgroup of $G_v$, which we denote $P_F$.
In particular, for any point $x\in\mathcal{B}_v$ if $F$ is the minimal facet containing $x$, then $P_x=P_F$.
To each parahoric $P_x$, Bruhat--Tits theory associates a smooth affine group scheme $\mathcal{G}_x$ over $\mathcal{O}_{K_v}$ such that the generic fiber $\mathcal{G}_x\times_{\mathcal{O}_{K_v}}K_v$ is $G_v$ and such that the $\mathcal{O}_{K_v}$ points give $P_x$.
In the sequel we refer to $\mathcal{G}_x$ as the \textbf{group scheme associated to the parahoric} $P_x$.

We may additionally reduce $\mathcal{G}_x$ over the residue field $f_v$, which we denote $\overline{\mathcal{G}_x}$.
More precisely, $\overline{\mathcal{G}_x}=\mathcal{G}_x\times_{\mathcal{O}_{K_v}}f_v$ and the corresponding map
$$P_x=\mathcal{G}_x(\mathcal{O}_{K_v})\to \overline{\mathcal{G}_x},$$
is surjective.
$\overline{\mathcal{G}_x}$ admits a Levi decomposition as $\overline{\mathcal{G}_x}=\overline{\mathcal{G}_x^{red}}~\overline{\mathcal{R}_u(\mathcal{G})}$, where $\overline{\mathcal{G}^{red}_x}$ is reductive and $\overline{\mathcal{R}_u(\mathcal{G})}$ is the unipotent radical.
Taking the quotient of $\overline{\mathcal{G}_x}$ by the unipotent radical gives a reductive algebraic group, the structure of which is dictated by Bruhat--Tits theory \cite[3.5.2]{Tits}.


\subsection{Combinatorial theorems}

Finally we require two combinatorial flavored theorems, a theorem of Green--Tao on progressions in the prime numbers and a theorem of Van der Waerden on progressions in the coloring of consecutive natural numbers.
The former is a result of Green--Tao in which they show that any positive upper density set of the prime numbers has arbitrarily long arithmetic progressions in the natural numbers.

\begin{thm}[\cite{GT}, Theorem 1.2]
Let $\pi(n)$ be the function that gives the cardinality of the set of prime numbers between $1$ and $n$.
If $A$ is a subset of the prime numbers such that
$$\limsup_{n\to\infty}\frac{\left|A\cap[1,n]\right|}{\pi(n)}>0,$$
then $A$ has arbitrarily long arithmetic progressions.
\end{thm}

\noindent The second theorem we will require is a result of Van der Waerden, which involves finding $k$-term arithmetic progressions in an $r$ coloring of the set $\{1,...,N\}$.
More specifically, Van der Waerden proved the following.

\begin{thm}[\cite{VDW}]\label{VDWThm}
Let $k,r$ be fixed natural numbers.
Then there exists an $N$ such that any $r$ coloring of the set $\{1,2,\dots,N\}$ has a monochromatic $k$-term progression.
\end{thm}

\noindent For any choice of $k,r$ in Theorem \ref{VDWThm}, the minimal $N$ such that this theorem holds will be denoted $W(r,k)$ and called a \textbf{Van der Waerden number}.
For ease of dialogue we will refer to these theorems in the sequel as the Green--Tao theorem and Van der Waerden's theorem, respectively.


\subsection{Structure of the paper.}

\noindent We give a brief roadmap for the proofs of Theorems \ref{ProgTheorem}, \ref{EveryPrimTheorem}, and \ref{CommInvTheorem}.
We first prove Theorems \ref{ProgTheorem} and \ref{EveryPrimTheorem} in the specific case when $\Gamma$ is the $\mathcal{O}_K$ points of $G$.
To accomplish this, we first define a class of hyperbolic elements called absolutely primitive elements in Section \ref{AbsPrimSection}, which will be instrumental in controlling the primitivity of the progressions we construct.
In Section \ref{primelengthsection}, we then use Bruhat--Tits theory to construct filtrations of subgroups in the lattice $\Gamma$.
Using these subgroups, we show that upon conjugating the absolutely primitive elements by suitable elements in the commensurator, we can control the power of the conjugated element which returns it to the lattice.
Furthermore, these elements will remain primitive after conjugation.
In Section \ref{ProgTheoremSection}, we will then use this and the conversion to geodesic lengths in the primitive length spectrum to prove Theorem \ref{ProgTheorem} for $\Gamma=G(\mathcal{O}_K)$.
In Section \ref{OtherThmSection}, we use a slight modification of the proof of Theorem \ref{ProgTheorem} combined with Lemma \ref{absprimlemma} to prove Theorem \ref{EveryPrimTheorem} for $\Gamma=G(\mathcal{O}_K)$.
We then prove Theorem \ref{CommInvTheorem} for lattices which are commensurable to this $\Gamma$.
By the definition of arithmeticity, this will complete the proofs of Theorems \ref{ProgTheorem} and \ref{EveryPrimTheorem} for all arithmetic locally symmetric spaces of classical type.


\section{Absolutely primitive elements}\label{AbsPrimSection}
\noindent In this section we introduce the notion of absolutely primitive elements of $\Gamma^{hyp}$.
Such elements will be an indispensable tool for guaranteeing primitivity in our construction of arithmetic progressions in the length spectrum.
In short, the definition of an absolutely primitive element is meant to encode when a hyperbolic element of $\Gamma$ has minimal possible eigenvalues in the sense that no other member of $\Gamma$ has eigenvalues which are $n$th roots of them.
Throughout the remainder, we use $p_\gamma(z)$ to denote the characteristic polynomial of $\gamma\in\Gamma$ and $\lambda_{\gamma,1},\dots,\lambda_{\gamma,n}$ to denote its eigenvalues.
Define $Q(x_1,\dots,x_n)$ by the following formula
$$Q(x_1,\dots,x_n)=\sum_{k=0}^n(-1)^{k}\Bigg(\sum_{\substack{i_j\in\{1,\dots,n\}\\ i_1<\dots<i_k}}x_{i_1}\dots x_{i_k}\Bigg) z^{n-k},$$
and notice that 
$$p_\gamma(z)=Q(\lambda_{\gamma,1},\lambda_{\gamma,2},\dots,\lambda_{\gamma,n})\in\mathcal{O}_K[z].$$
If $\mathcal{P}$ denotes the set of all characteristic polynomials of elements of $\Gamma^{hyp}$, then we say that any $\gamma\in\Gamma^{hyp}$ is an \textbf{absolutely primitive element} if it satisfies the condition that
$$\max\left\{k\in\N\mid Q(\lambda^{1/k}_{\gamma,1},\lambda^{1/k}_{\gamma,2},\dots,\lambda^{1/k}_{\gamma,n})\in\mathcal{P}\right\}=1.$$
Using the correspondence between hyperbolic elements and closed geodesics, we say that $c_\gamma$ is an absolutely primitive geodesic if $\gamma$ is absolutely primitive.
With this terminology in mind, we have the following lemma.

\begin{lemma}\label{absprimlemma}
Absolutely primitive elements are primitive in $\Gamma$.
Moreover, for any $\gamma\in\Gamma^{hyp}$, there exists an absolutely primitive element $\mu$ and an $m\in\N$ such that $\lambda_{\gamma,i} =\lambda_{\mu,i}^m$ for all $i$.
\end{lemma}

\begin{proof}
The first claim is immediate from the definition.
For the second, let $\gamma$ be a fixed hyperbolic element which is not absolutely primitive.
We claim that there is a maximum element $m$ of the set
$$\left\{k\in\N\mid Q(\lambda^{1/k}_{\gamma,1},\lambda^{1/k}_{\gamma,2},\dots,\lambda^{1/k}_{\gamma,n})\in\mathcal{P}\right\}.$$
To justify this let $K_\gamma=K(\lambda_{\gamma,1},\lambda_{\gamma,2},\dots,\lambda_{\gamma,n})$.
As $\det(\gamma)=1$, each $\lambda_{\gamma,i}$ is a unit in $\mathcal{O}_{K_\gamma}$.
Dirichlet's unit theorem then yields that each $\lambda_{\gamma,i}$ is a product of powers of the fundamental units of $\mathcal{O}_{K_\gamma}^1$.
Hence for some natural number $k\in\N$ there is an $i$ such that $\lambda_{\gamma,i}^{1/k}$ no longer lies in $\mathcal{O}_{K_\gamma}^1$. 
This implies that for such $k$,
$$[K(\lambda^{1/k}_{\gamma,1},\lambda^{1/k}_{\gamma,2},\dots,\lambda_{\gamma,n}^{1/k}):K_\gamma]>1.$$
Moreover, repeating this procedure we eventually find that there exists some $N\in\N$ such that for all $k\ge N$,
$$[K(\lambda^{1/k}_{\gamma,1},\lambda^{1/k}_{\gamma,2},\dots,\lambda_{\gamma,n}^{1/k}):K]>n!.$$
It follows that $Q(\lambda^{1/k}_{\gamma,1},\lambda^{1/k}_{\gamma,2},\dots,\lambda^{1/k}_{\gamma,n})$ cannot lie in $\mathcal{P}$ for any such $k$.
Consequently such an $m$ necessarily exists.

Now let $\mu\in\Gamma^{hyp}$ be an element that obtains this maximum, i.e. has eigenvalues $\lambda_{\mu,1},\lambda_{\mu,2},\dots,\lambda_{\mu,n}$ such that
$$p_\mu(z)=Q(\lambda_{\mu,1},\lambda_{\mu,2},\dots,\lambda_{\mu,n})=Q(\lambda^{1/m}_{\gamma,1},\lambda^{1/m}_{\gamma,2},\dots,\lambda^{1/m}_{\gamma,n}).$$
Then $\mu$ is the requisite absolutely primitive element.
\end{proof}

\noindent We briefly remark that our definition of absolute primitivity should be considered a generalization of that found in Lafont--McReynolds \cite[p. 16]{LaMcR}.
The definition given there is that an absolutely primitive element is one which has the fundamental units of $K_\gamma$ as its eigenvalues.
Indeed in the cases considered therein, the two definitions are equivalent.
We briefly sketch this equivalence now.

The case they consider is when $K=\Q$ and $\Gamma=\SL(2,\Z)$.
Hence if $\gamma$ is a hyperbolic element of $\SL(2,\Z)$, $K_\gamma=\Q(\lambda_{\gamma,1})$ is a real quadratic extension of $\Q$.
Consequently $\lambda_{\gamma,1}\in\mathcal{O}_{K_\gamma}^1$ and $\lambda_{\gamma,2}=\lambda_{\gamma,1}^{-1}$ is the Galois conjugate of $\lambda_{\gamma,1}$.
The condition of being absolutely primitive then reduces to requiring that
$$\max\left\{k\in\N\mid (\lambda^{1/k}_{\gamma,1}+\lambda^{-1/k}_{\gamma,1})\in\Z\right\}=\max\left\{k\in\N\mid \lambda^{1/k}_{\gamma,1}\in\mathcal{O}_{K_\gamma}\right\}=1.$$
By Dirichlet's unit theorem, $\lambda_{\gamma,1}$ is the power of some fundamental unit of $\mathcal{O}_{K_\gamma}^1$, so the condition of being absolutely primitive implies that $\lambda_{\gamma,1}$ is a fundamental unit of $\mathcal{O}^1_{K_\gamma}$.
Moreover, one can check that any real quadratic extension $K$ of $\Q$ embeds into $\GL(2,\Q)$ and subsequently $\mathcal{O}_K^1$ embeds into $\SL(2,\Z)$.
Thus for any such $K$, we can always find some $\gamma\in\Gamma^{hyp}$ that has absolutely primitive eigenvalues.

The trouble with using a definition in terms of fundamental units in the general case, stems from the fact that a priori $K_\gamma$ may have degree $n!$ over $K$.
Consequently one cannot use an embedding of $K_\gamma$ into $G$ to produce elements with eigenvalues being merely a product of fundamental units.
In fact, recent results have shown that this issue is generic in the sense that for most elements $\gamma\in\Gamma^{hyp}$ this will be an issue (see for instance \cite{PrasadRap2}, \cite{LubRos}, \cite{JKZ}, or \cite{GorNev}), hence the restriction we impose.


\section{Controlling the lengths of the associated primitive geodesics}\label{primelengthsection}

\noindent In this section, we give the foundational work for producing subsets of the primitive length spectrum.
Letting $\Gamma=G(\mathcal{O}_K)$, we are interested in analyzing the function
\begin{align*}
n:\Gamma\times\Comm(\Gamma)&\to \N\\
(\gamma,\eta)&\mapsto \min\{j\in\N\mid \eta\gamma^j\eta^{-1}\in\Gamma\}.
\end{align*}
We show that for a certain class of elements $\eta\in\Comm(\Gamma)$, which we call admissible, the value of $n(\gamma,\eta^{r+1})$ can be controlled in terms of the value of $n(\gamma,\eta^r)$.
Indeed using the convention that $S$ denotes the set of finite, unramified places of $K$ such that $G_v$ is inner and split, we devote the entirety of this section to proving the following result.

\begin{thm}\label{primepowerseqthm}
For any $v\in S$, any $\gamma\in\Gamma^{hyp}$, and any admissible $\eta$, the following hold:
\begin{enumerate}
\item If $r$ is large enough then $n(\gamma,\eta^{r+1})=q_v^{\epsilon_r\omega_\eta}n(\gamma,\eta^{r})$, where $\epsilon_r\in\{0,1\}$ and $\omega_\eta\mid\beta_\eta$.
\item If $\gamma$ is absolutely primitive, then $\eta^r\gamma^{n(\gamma,\eta^r)}\eta^{-r}$ is a primitive element of $\Gamma$  for every $r$.
\end{enumerate}
\end{thm}

\noindent We will define admissible elements and $\beta_\eta$ in the subsequent text, however first we briefly give the idea of the proof of Theorem \ref{primepowerseqthm}.
Given a hyperbolic $\gamma$, we will choose $\eta\in\Comm(\Gamma)$ so that conjugation by $\eta$ multiplies the matrix coefficients of $\gamma$ in a controlled way.
We then use the existence of sequences of finite index subgroups which filter the lattice $\Gamma$ to relate the value of $n(\gamma,\eta^r)$ to the minimal power of $\gamma$ which lies in one of these filtering subgroups.
This will allow us to relate the values of $n(\gamma,\eta^{r})$ and $n(\gamma,\eta^{r+1})$ by using the corresponding relationship between powers of $\gamma$ landing in successive steps of this filtration.
The construction of the filtration will use natural maps to finite groups in a way such that we obtain precise control on the latter quantity.

Before proceeding to the proofs, we now define the building blocks of our admissible $\eta$.
To this end, let
$$\xi_{v,k}=\diag(1,\dots,1,v,v^{-1},1,\dots,1),$$
where $v$ is in the $k$th spot.
By convention we define $\xi_{v,n}$ to be the matrix
$$\xi_{v,n}=\diag(v^{-1},1,\dots,1,v).$$
Throughout the entirety of this section, $\eta$ will always be an element of the form
$$\eta=\xi_{v,1}^{\alpha_1}\cdot\xi_{v,2}^{\alpha_2}\cdot~\dots~\cdot\xi_{v,n}^{\alpha_n},$$
where $\alpha_i\in\Q$.
For each classical type, we always assume the $\alpha_i$ are chosen so that under the isomorphism afforded by restriction of scalars, $\eta\in\mathcal{G}$.
Conjugating a matrix $A=(a_{ij})$ by $\eta$ gives $\eta A\eta^{-1}=(b_{ij})$, where $b_{ij}=v^{-\beta_{ij}}a_{ij}$ and $\beta_{ij}\in\Q$.
We call $\eta$ \textbf{admissible} if the following two conditions are satisfied:
\begin{enumerate}
\item $\beta_{ij}\in\Z$ for all $1\le i,j\le n$.
\item The $\beta_{ij}$ are constant on each root subgroup.
Namely for each root $\alpha$, if $i,j$ and $i',j'$ are such that the $i,j$th and $i',j'$th coefficients in $U_\alpha$ are non-diagonal and non-zero, then $\beta_{ij}=\beta_{i'j'}$.
\end{enumerate} 
For such admissible $\eta$, we also define
$$\beta_\eta=\lcm\{\beta_{ij}\mid \beta_{ij}>0\}.$$
This is the quantity mentioned in (1) of Theorem \ref{primepowerseqthm}.
By construction, any admissible $\eta$ is contained in the commensurator of $\Gamma$.


\subsection{Construction of filtering subgroups.}\label{buildinglemmasection}

We now begin with the construction of our filtering subgroups of $\Gamma$.
These will arise as kernels of natural maps to finite groups, afforded by applying Bruhat--Tits theory to the $G_v$ over the local fields $K_v$.
Recall from above that $S$ denotes the set of finite, unramified places of $K$ such that $G_v$ is inner and split.
We begin with a building theoretic lemma.

\begin{lemma}\label{buildinglemma}
Let $i,j\in\{1,\dots,n\}$ with $i\neq j$ and let $\alpha$ be the unique root such that $U_\alpha$ has a non-zero coefficient in the $i,\!j$th entry.
Then for any $v\in S$, there exists a point $x\in\mathcal{B}_v$ and a parahoric $P_x$ with associated group scheme $\mathcal{G}_x$ such that either
$$\overline{\mathcal{G}_x^{red}}\cong\GL_2(f_v)\times\left(\prod_{z=1}^{\rk_{K_v}(G_v)-2}\GL_1(f_v)\right),$$
or
$$\overline{\mathcal{G}_x^{red}}\cong\SL_2(f_v)\times\left(\prod_{z=1}^{\rk_{K_v}(G_v)-1}\GL_1(f_v)\right).$$
Furthermore, the only non-zero entries of $\overline{\mathcal{G}_x^{red}}$ which are not on the diagonal are those with the same entries as those of the root subgroups $U_\alpha$ and $U_{-\alpha}$.
\end{lemma}

\begin{proof}
As $G_v$ is inner and split, we know that $\rk_{K_v}(G_v)$ is the same as the absolute rank of $G$ and all of its roots are defined over $K_v$.
Let $\mathcal{B}_v$ be the associated building.
For the first claim we are only concerned with the isomorphism type of $\overline{\mathcal{G}_x^{red}}$ and this is invariant under choice of chamber, so it suffices to simply consider the closure of a single chamber in a single apartment.
The closure $\overline{C}$ of such a chamber is a $d$-dimensional simplicial complex.
Now let $x$ be any point such that the maximal facet containing $x$ in $\overline{C}$ is a $(d-1)$-simplex.
Then $x$ is contained in the vanishing hyperplane of a single root of the local Dynkin diagram.
By our assumptions on $G_v$, the local Dynkin diagram is simply the extended Dynkin diagram carrying a trivial Galois action.
Bruhat--Tits theory then gives that the Dynkin diagram of $\overline{\mathcal{G}_x^{red}}$ is obtained by removing all nodes except for one \cite[Section 3.5.2]{Tits}.
Combining this with the fact that $\rk_{\mathbb{F}_v}(\overline{\mathcal{G}_x^{red}})=\rk_{K_v}(G_v)$, we obtain the desired conclusion on the isomorphism type of $\overline{\mathcal{G}_x^{red}}$.
The claim on root subgroups follows from the fact that, in each apartment, every root produces a collection of vanishing hyperplanes and each of these hyperplanes gives rise to several $(d-1)$-simplices so it suffices to choose a point on one of these.
\end{proof}

\noindent We remark that the specific choice of $x$ in Lemma \ref{buildinglemma} is not unique.
Consequently when we refer to it in the subsequent text, any such choice of $x$ will suffice.
Assuming the notation of Lemma \ref{buildinglemma}, we now fix some notation for the next lemma as well as forthcoming sections.
Since $P_x=\mathcal{G}_x(\mathcal{O}_{K_v})$, we may define natural reduction maps
$$\xymatrixrowsep{.05in}
\xymatrix{\pi^{(r)}:P_x\ar[r]&\overline{\mathcal{G}_x^{red}}(\mathcal{O}_{K_v}/v^r\mathcal{O}_{K_v})\\
\overline{\pi}^{(r)}:\overline{\mathcal{G}_x^{red}}(\mathcal{O}_{K_v}/v^{r}\mathcal{O}_{K_v})\ar[r]&\overline{\mathcal{G}_x^{red}}(\mathcal{O}_{K_v}/v^{r-1}\mathcal{O}_{K_v})}$$
with the relationship $\pi^{(r)}=\overline{\pi}^{(r+1)}\circ\pi^{(r+1)}$.
With this notation in mind, we prove the following lemma.

\begin{lemma}\label{exponentlemma}
For any $B\in P_x$ and any sufficiently large $r$, if $\pi^{(r)}(B)$ is in the kernel of $\overline{\pi}^{(r)}$ then $\pi^{(r)}(B)$ has order $1$ or $q_v$.
\end{lemma}

\begin{proof}
For ease of notation, let $q=q_v=p^e$ for some $e>0$.
Throughout the proof, $\id_{i}$ will denote the identity in $\SL(n,\Z/p^i\Z)$ and
$$\psi_{i+1}:\SL(n,p^{i+1}\Z)\to\SL(n,\Z/p^i\Z),$$
will denote the natural reduction maps.
The order of any element in $\ker(\psi_{i})$ is easily seen to be $1$ or $p$.
Note that we may immediately reduce to proving the corresponding result for $\SL(n,\Z/q^s\Z)$.
Indeed, the conclusion for $\SL(n,\Z/q^s\Z)$ implies the conclusion for reduction maps between its subgroups and hence the $\overline{\mathcal{G}_x^{red}}(\mathcal{O}_{K_v}/v^{s}\mathcal{O}_{K_v})$.
With this in mind, we will need the following claim.

\begin{claim}\label{exponentclaim}
Let $B\in \SL(n,\Z/p^{i+2}\Z)$ such that $\id_{i+1}\neq \psi_{i+2}(B)\in\ker(\psi_{i+1})$.
Then $B^p\neq\id_{i+2}$, so long as $i\neq 1$ and $p\neq 2$.
\end{claim}

\noindent Momentarily assuming Claim \ref{exponentclaim}, we complete the proof.
As we are only concerned with sufficiently large $r$, we may ignore the case when $p=2$ and $i=1$.
After identification we have the natural reduction maps
$$\overline{\pi}^{(s+1)}:\SL(\Z/q^{s+1}\Z)\to \SL(\Z/q^s\Z),$$
which we may write as
$$\overline{\pi}^{(s+1)}= \psi_{se+1}\circ\dots\circ \psi_{(s+1)e-1}\circ \psi_{(s+1)e}~.$$
If $\pi^{(s+1)}(B)$ is a non-trivial element of $\ker(\overline{\pi}^{(s+1)})$ such that $(\psi_{se+2}\circ\dots\circ \psi_{(s+1)e-1}\circ \psi_{(s+1)e})(\pi^{(s+1)}(B))$ is also non-trivial, then by applying Claim \ref{exponentclaim} repeatedly to the maps $\psi_{(s+1)e-j}$ for $0\le j< e$ we see that the order of $B$ must be $q$.
Otherwise, taking $s'=s+1$ will force that $(\psi_{s'e+2}\circ\dots\circ \psi_{(s'+1)e-1}\circ \psi_{(s'+1)e})(\pi^{(s'+2)}(B))$ is non-trivial and so the same argument implies $B$ must have order $q$.
Hence the conclusion holds for $\overline{\pi}^{(s+1)}$ or $\overline{\pi}^{(s'+1)}$.
An iterated application of Claim \ref{exponentclaim} then implies that the same conclusion must also hold for all integers larger than $s'$.
This completes the proof.
\end{proof}

\begin{proof}[Proof of Claim \ref{exponentclaim}]
Let $B\in \SL(n,\Z/p^{i+2}\Z)$.
We prove the contrapositive, namely that $B^p=\id_{i+2}$ and $\psi_{i+2}(B)\in\ker(\psi_{i+1})$ implies that $\psi_{i+2}(B)=\id_{i+1}$.
For ease of notation, choose any lift of $B$ to an element of $\SL(n,\Z)$, i.e. choose $\widetilde{B}$ such that $\widetilde{B}\equiv B\pmod{p^{i+2}}$.
As $\psi_{i+2}(B)\in\ker(\psi_{i+1})$, we may write $\widetilde{B}=\id+p^iC$, where $\id$ is the identity matrix in $\SL(n,\Z)$.
We will show that $p\mid C$ unless $p=2$ and $i=1$.

Since $B^p\equiv\id_{i+2}\in\SL(n,\Z/p^{i+2}\Z)$, we may also write $\widetilde{B}$ as
$$\widetilde{B}^p=\id+p^{i+2}C',$$
for some $C'$.
The matrix binomial theorem then gives
\begin{align*}
\widetilde{B}^p&=(\id+p^iC)^p,\\
&=\id+p^{i+1}C+\sum_{k=1}^{p-2}\binom{p}{k}(\id)^k(p^iC)^{p-k}+p^{ip}C^p,
\end{align*}
and so 
$$p^{i+2}C'=p^{i+1}C+\sum_{k=1}^{p-2}\binom{p}{k}(\id)^k(p^iC)^{p-k}+p^{ip}C^p.$$
As we are outside of the case when $p=2$ and $i=1$, we know that $ip\ge i+2$.
Hence $p\mid C$ completing the proof of Claim \ref{exponentclaim} and consequently Lemma \ref{exponentlemma}.
\end{proof}

\noindent Examining the proof of Lemma \ref{exponentlemma} immediately yields the following.

\begin{sch}\label{exponentcorollary}
Let $\gamma\in P_x$ such that $\pi^{(N)}(\gamma)\in\ker(\overline{\pi}^{(N)})$ has order $q_v$.
Then $\pi^{(N+k)}(\gamma)$ has order $q_v$ for any $k\ge 0$.
\end{sch}


\subsection{Controlling the function $n(-,-)$}\label{controlsection}

\noindent In this section, we use the results of Section \ref{buildinglemmasection} to prove (1) in Theorem \ref{primepowerseqthm}.
More specifically, we use the filtering subgroups given as the kernels of each $\pi^{(r)}$ intersected with $\Gamma$ to control the function $n(\gamma,-)$.
This allows us to prove the following.

\begin{prop}\label{reductionprop}
Fix $\gamma\in\Gamma^{hyp}$ and $v\in S$.
Then for any admissible $\eta$ and sufficiently large $r$
$$n(\gamma,\eta^{r+1})=q_v^{\epsilon_r\omega_\eta}n(\gamma,\eta^r),$$
where $\epsilon_r\in\{0,1\}$ and $\omega_\eta\mid\beta_\eta$.
\end{prop}

\begin{proof}
We write $\gamma=(\gamma_{ij})$.
As $\Gamma$ is the $\mathcal{O}_K$ points of $G$, we see that $\eta\gamma\eta^{-1}\in\Gamma$ if and only if $v^{\beta_{ij}}\mid\gamma_{ij}$, whenever $\beta_{ij}>0$.
An equivalent condition to this can be given in terms of the subgroups afforded by Lemma \ref{buildinglemma}.
To this end, let
$$T=\{(i,j)\mid \beta_{ij}>0\}.$$
For any $(i,j)$ in $T$, let $x_{ij}$ be the point in $\mathcal{B}_v$ furnished by Lemma \ref{buildinglemma}, let $P_{ij}=P_{x_{ij}}$ be the corresponding parahoric, and let $\mathcal{G}_{ij}$ be the group scheme associated to $P_{ij}$.
For each $\mathcal{G}_{ij}$, we denote the natural surjections from Section \ref{buildinglemmasection} as $\pi^{(r)}_{ij}$ and $\overline{\pi}^{(r)}_{ij}$.
We claim that there is some number $Z$ such that for all multiples $z$ of $Z$, $\gamma^z\in P_{ij}$ for all $1\le i, j\le n$ with $i\neq j$.
Indeed by construction of the $P_{ij}$ in Lemma \ref{buildinglemma}, we know that $P_{ij}$ is contained in some maximal hyperspecial parahoric $P^{hyp}_{ij}$.
This follows from the fact that, in the inner, split case, there is at least one hyperspecial vertex in the closure of every $(d-1)$--simplex of each chamber.
Moreover, this implies that there is an embedding of $\Gamma$ into $P^{hyp}_{ij}$.
We then have the natural reduction maps
$$\xymatrixrowsep{.11in}\xymatrix{P_{ij}^{hyp}\ar[r]^{\pi_{ij}^{hyp}}\ar@{}[d]|{\rotsubset}&\overline{\mathcal{G}}\ar@{}[d]|{\rotsubset}\\
P_{ij}\ar[r]^{\pi_{ij}}&\overline{\mathcal{G}_{ij}^{red}}}$$
where $\overline{\mathcal{G}}$ is either $\SL_n(f_v)$, $\Sp_n(f_v)$, or $\SO_n(f_v)$.
As $\overline{\mathcal{G}}$ is finite, some power of $\gamma$ becomes trivial and we denote this power by $z_{ij}$.
Letting $Z=\max\{z_{ij}\}$ gives the desired claim.
Furthermore, one can easily see that $n(\gamma,\eta)=Z$.

From this claim it is now clear that $\eta^r\gamma^z\eta^{-r}\in\Gamma$ if and only if
$$\pi_{ij}^{(r\beta_{ij})}(\gamma^z)=\overline{\id}\in\overline{\mathcal{G}_{ij}^{red}}(\mathcal{O}_{K_v}/v^{r\beta_{ij}}\mathcal{O}_{K_v}),$$
for all $(i,j)$ in $T$ and $z$ some multiple of $Z$.
To control the function $n(\gamma,\eta^r)$ we use this latter criteria, which implies that $z=n(\gamma,\eta^r)$ if and only if the following three conditions are satisfied
\begin{itemize}
\item $z$ is some multiple of $Z$,
\item $\pi_{ij}^{(r\beta_{ij})}(\gamma^z)=\overline{\id}\in\overline{\mathcal{G}_{ij}^{red}}(\mathcal{O}_{K_v}/v^{r\beta_{ij}}\mathcal{O}_{K_v})$, for all $(i,j)\in T$,
\item $z=\min\{k\in\N\mid \pi_{ij}^{(r\beta_{ij})}(\gamma^k)=\overline{\id}\in\overline{\mathcal{G}_{ij}^{red}}(\mathcal{O}_{K_v}/v^r\mathcal{O}_{K_v})\}$, for some $(i,j)\in T$.
\end{itemize}
Lemma \ref{exponentlemma} now implies that there exists some $R$ such that for $r\ge R$, the order of $\pi^{(r)}(\gamma^z)$ in $\ker(\overline{\pi}_{ij}^{(r)})$ is $1$ or $q_v$.
Consequently, Scholium \ref{exponentcorollary} implies that for fixed $(i,j)$ in $T$ and large enough $r$, the order of any element in the kernel of the function
$$\overline{\pi}_{ij}^{((r+1)\beta_{ij})}=\overline{\pi}_{ij}^{(r\beta_{ij}+1)}\circ\dots\circ\overline{\pi}_{ij}^{((r+1)\beta_{ij}-1)},$$
is $1$ or $q_v^{\beta_{ij}}$.
Scholium \ref{exponentcorollary} also gives that for large enough $r$, the minimality condition is satisfied for some $(i,j)$ in $T$ (unless of course $\pi_{ij}^{(r)}(\gamma^z)$ is trivial for all $r$ and all $(i,j)\in T$).
Thus the order of any element jointly in the kernels of each $\pi_{ij}^{(r\beta_{ij})}$ is the same as the order of any element in the kernel of
$$\prod_{(i,j)\in T}\overline{\pi}_{ij}^{(r+1)}\circ\dots\circ\overline{\pi}_{ij}^{(r+\beta_{ij})}.$$
Clearly this is $1$ or $q_v^{\omega_\eta}$ for some $\omega_\eta$ dividing $\beta_\eta$.
We now see that for $r\ge R$, $n(\gamma,\eta^{r+1})/n(\gamma,\eta^{r})$ is $q_v^{\epsilon_r\omega_\eta}$, completing the proof.
\end{proof}


\subsection{Guaranteeing primitivity.}

\noindent We now turn to proving (2) in Theorem \ref{primepowerseqthm}.
In particular we show that given an absolutely primitive $\gamma$, the elements $\theta_{r}=\eta^r\gamma^{n(\gamma,\eta^r)}\eta^{-r}$ are indeed primitive.
To this end we need the following generalization of a statement from \cite[p. 20]{LaMcR} to our definition of absolutely primitive.

\begin{lemma}
If $\gamma$ is absolutely primitive then the elements $\theta_{r}$ are primitive in $\Gamma$ for all $r$.
\end{lemma}

\begin{proof}
For fixed $r$, we write $\theta=\theta_{r}$ for notational convenience.
Assume that $\theta=\mu^m$ for some $\mu\in\Gamma$, then we show that $m=1$.
Upon diagonalizing $\theta$ by some $D\in\GL(n,\C)$, we see that
$$
\begin{pmatrix}
\lambda_{\gamma,1}^{n(\gamma,\eta^r)}&&\\
&\ddots&\\
&&\lambda_{\gamma,n}^{n(\gamma,\eta^r)}
\end{pmatrix}
=
\begin{pmatrix}
\lambda_{\theta,1}&\\
&\ddots&\\
&&\lambda_{\theta,n}
\end{pmatrix}
=
\begin{pmatrix}
\lambda_{\mu,1}^m&&\\
&\ddots&\\
&&\lambda_{\mu,n}^{m}
\end{pmatrix}.$$
We claim that the absolute primitivity condition guarantees that all of the $\lambda_{\mu,i}$ are some integral power of the $\lambda_{\gamma,i}$.
For contradiction assume otherwise, i.e. that $\gcd(m,n(\gamma,\eta^r))\neq m$.
First note that if
$$\gcd(m,n(\gamma,\eta^r))=n(\gamma,\eta^r)\neq m,$$
then $\lambda_{\mu,i}=\lambda_{\gamma,i}^{1/k}$ for some $k>1$.
This contradicts the absolute primitivity of $\gamma$, hence we may assume that $\gcd(m,n(\gamma,\eta^r))=d<m$.
As a result, $sm+tn(\gamma,\eta^r)=d$ for some integers $s$ and $t$.
Since $\mu$ and $\gamma$ have the same eigenvectors, $\gamma^{s}\mu^{t}$ is a member of $\Gamma^{hyp}$ with eigenvalues
$$\lambda^s_{\gamma,i}\lambda^{t}_{\mu,i}=\lambda^s_{\gamma,i}\lambda^{tn(\gamma,\eta^r)/m}_{\gamma,i}=\lambda^{d/m}_{\gamma,i}=\lambda^{1/k}_{\gamma,i},$$
for some $k>1$.
Once again this contradicts the absolute primitivity of $\gamma$.
Consequently $m\mid n$, so we may write $n=mk$ for some $k\in \N$.
However construction of the function $n(-,-)$ implies we are now done since
$$\mu^m=\eta^r\gamma^{n(\gamma,\eta^r)}\eta^{-r}=(\eta^r\gamma^k\eta^{-r})^m.$$
As $\mu\in\Gamma$ this gives that $m=1$ and $k=n$.
\end{proof}


\section{The proof of Theorem \ref{ProgTheorem} for $\Gamma=G(\mathcal{O}_K)$.}\label{ProgTheoremSection}

\noindent We now turn to the task of constructing arithmetic progressions in the primitive length spectrum for the specific case of $\Gamma=G(\mathcal{O}_K)$.
To do so, we first show that we may use Theorem \ref{primepowerseqthm} to construct subsets of a very particular form in the primitive length spectrum.
Namely, we prove the following.

\begin{prop}\label{primepowerlengprop}
For every absolutely primitive $\gamma\in\Gamma^{hyp}$ and every $v\in S$, there exists subsets of the form
$$\{\ell,~q_v\ell,~q_v^{2}\ell,\dots\}\subset\mathcal{L}_p(M),$$
in the primitive length spectrum, where $\ell=\leng(\gamma)$.
\end{prop}

\noindent Indeed, this will follow from Theorem \ref{primepowerseqthm} once we show that for each absolutely primitive $\gamma$, there is an admissible $\eta$ such that $\beta_\eta=1$ and such that the $n(\gamma,\eta^r)$ is unbounded.
However, note that this does not directly imply Theorem \ref{ProgTheorem} as these subsets are fairly far from being genuine arithmetic progressions.
In the latter part of this section, we go on to patch such subsets together by proving the following.

\begin{prop}\label{progglueinglemma}
Let $\gamma$ be absolutely primitive and $c_\gamma$ the associated geodesic with $\ell=\leng(c_\gamma)$.
Then for any $v,v'\in S$, there exists some $C\in\N$ such that
$$\{Cq_v^{t}q_{v'}^{t'}\ell\}_{(t,t')\in\N^2}\subset\mathcal{L}_p(M),$$
where our notation means that $t$ and $t'$ may vary over all natural numbers.
\end{prop}

\noindent An iterative application of Proposition \ref{progglueinglemma} for well chosen elements of $S$ will then yield Theorem \ref{ProgTheorem} for $G(\mathcal{O}_K)$.
We first prove Proposition \ref{primepowerlengprop}, which follows immediately from the following.

\begin{lemma}\label{unboundedlemma}
Let $\gamma\in\Gamma^{hyp}$ be absolutely primitive, let $v\in S$, and let $\ell=\leng(c_\gamma)$.
Then there exists an absolutely primitive element $\widehat\gamma\in\Gamma^{hyp}$, independent of $v$, such that $\leng(c_{\widehat\gamma})=\ell$ and such that there is a choice of admissible $\eta$ with the properties that $\beta_\eta=1$ and $n(\widehat\gamma,\eta^{r})$ is unbounded as $r\to\infty$.
\end{lemma}

\begin{proof}
First we find $\widehat\gamma$ so that no power of $\widehat\gamma$ becomes a diagonal element.
Indeed, if this is already true of $\gamma$ then let $\widehat\gamma=\gamma$, otherwise let $\widehat\gamma$ be a conjugate of $\gamma$ by some element in $\Gamma$ such that no power is diagonal.
Now write $\widehat\gamma=(\widehat\gamma_{ij})$ and use superscripted notation to denote powers of $\widehat\gamma$, namely $\widehat\gamma^z=(\widehat\gamma^{(z)}_{ij})$.
By assumption, there exists some $k,k'\in\{1,\dots,n\}$ such that $k\neq k'$ and $\widehat\gamma_{kk'}\neq 0$.
Moreover as $\gamma$ is hyperbolic, so is $\widehat\gamma$ and consequently we may choose $k$ and $k'$ with some additional useful restrictions for certain classical types of $G$.
Namely, we have the following restrictions:
\begin{itemize}[leftmargin=5.5mm]
\item[] \textbf{$G$ is of type $^2A$, $B$, or $^1D$.} We may choose $k$ and $k'$ so that not both $k$ and $k'$ are less than or equal to $q-p$.
Indeed if this were not possible, then $\gamma$ would be an orthogonal matrix and consequently cannot be hyperbolic.
Furthermore such $k,k'$ can be chosen so that $k\neq 2q-k'+1$, as otherwise some power of $\gamma$ would be diagonal.
\item[]
\item[] \textbf{$G$ is of type $C$.} We may choose $k$ and $k'$ so that they are not mutually in the set $\{1,\dots,p,p+q+1,\dots,2p+q\}$ or its complement in $\{1,\dots,n\}$.
Again if this were not possible, then $\gamma$ could not be hyperbolic.
Furthermore such $k,k'$ can be chosen so that $k\neq 2q-k'+1$, as otherwise some power of $\gamma$ would be diagonal.
\end{itemize}
Note that length as well as the property of being absolutely primitive are conjugacy invariant so $\widehat\gamma$ is absolutely primitive and
$$\leng(c_\gamma)=\leng(c_{\widehat\gamma}).$$
It is also clear that our construction of $\widehat{\gamma}$ never used any reference to $v\in S$ and so is independent of that.
There are now several ways to build an admissible $\eta$ using Theorem \ref{primepowerseqthm}.
We give the most consistent choices herein.
\begin{itemize}[leftmargin=5.5mm]
\item[] \textbf{$G$ is of type $^2A$, $B$, or $^1D$.} We define
$$\eta=\diag(1,\dots,1,v^{m_{q-p+1}},\dots,v^{n}),$$
where the number of $1$s at the beginning of $\eta$ is $q-p$.
If $k\le p$, then the $m_i$ are chosen so that $m_{k'}=1$, $m_{2q-k'+1}=-1$, and $m_i=0$ for all $i\neq k',~2q-k'+1$.
Otherwise, we may choose $m_i$ so that $m_{k}=-1$, $m_{2q-k+1}=1$, and $m_i=0$ for all $i\neq k,~2q-k+1$.
\item[]
\item[] \textbf{$G$ is of type $C$.} Let
$$\eta=\diag(v^{m_1},\dots,v^{m_n}).$$
If $k'\le n/2$ then let $m_{k'}=1$, $m_{n/2+k'}=-1$, and the other $m_i=0$.
Otherwise, let $m_{k'}=1$, $m_{k'-n/2}=-1$, and the other $m_i=0$.
\item[]
\item[] \textbf{$G$ is of type $A_1$.} Then we define
$$\eta=\diag(v^{m_1/2},v^{m_2/2}),$$
where $(m_1,m_2)=(1,-1)$ if $(k,k')=(1,2)$ and $(m_1,m_2)=(-1,1)$ if $(k,k')=(2,1)$.
\item[]
\item[] \textbf{$G$ has any other type.} Then
$$\eta=\diag(v^{m_1},\dots,v^{m_n}),$$
where the $m_i$ are taken so that $m_k=1$, $m_{n-k+1}=-1$, and the other $m_i=0$.

\end{itemize}
It is now clear by computation in each of these cases that $\beta_\eta=1$ and that $\eta$ is admissible.
Furthermore, each admissible $\eta$ multiplies the $kk'$th coefficient of $\widehat\gamma$ by $v^{-1}$.
To see the claim that $n(\widehat\gamma,\eta^r)$ are unbounded, note that if $n(\widehat\gamma,\eta^r)$ were bounded as $r$ goes to infinity then for all $r$ larger than some fixed natural number $R$ we would have $z=n(\widehat\gamma,\eta^{r})=n(\widehat\gamma,\eta^{R})$.
Consequently, $\widehat\gamma$ would be in the kernel of $\pi_{ij}^{(r\beta_{ij})}$ for all $r>R$.
This however would imply that some power of $\widehat\gamma$ is diagonal, contradicting our choice of $\widehat\gamma$.
\end{proof}

\noindent This completes the proof of Proposition \ref{primepowerlengprop}.
We now show that these subsets can be glued together to give Proposition \ref{progglueinglemma}.

\begin{proof}[Proof of Proposition \ref{progglueinglemma}]
For fixed $\gamma$, let $\eta_v$, $\eta_{v'}$, and $\widehat\gamma$ denote the elements furnished by Lemma \ref{unboundedlemma} and define $\eta_{r,r'}=\eta_v^r\eta_{v'}^{r'}$.
Recall that $\beta_{\eta_v}=\beta_{\eta_{v'}}=1$.
By construction, there exists some $R\in\N$ such that for any $r\ge R$, $n(\widehat\gamma,\eta_v^{r+1})=q_vn(\widehat\gamma,\eta_v^{r})$ and $n(\widehat\gamma,\eta_{v'}^{r+1})=q_{v'}n(\widehat\gamma,\eta_{v'}^{r})$.
Using the Chinese remainder theorem, we obtain the relationship
$$n(\widehat\gamma,(\eta_v\eta_{v'})^R)=\lcm\{n(\widehat\gamma,\eta_v^{R}),n(\widehat\gamma,\eta_{v'}^{R})\}.$$
Calling this quantity $C$, we see that for any $r,r'\ge R$
$$n(\widehat\gamma,\eta_{r,r'})=n(\widehat\gamma,\eta_v^r\eta_{v'}^{r'})=Cq_v^{(r-R)}q_{v'}^{(r'-R)}.$$
The translation to the primitive length spectrum follows from (2) of Theorem \ref{primepowerseqthm} and the fact that $\ell=\leng(c_\gamma)=\leng(c_{\widehat\gamma})$.
\end{proof}

\noindent Before we proceed to the proof of Theorem \ref{ProgTheorem}, we need one final ingredient.

\begin{lemma}\label{PrimeDensLemma}
The set $S$ has positive natural density in the set of places of $K$.
\end{lemma}
\begin{proof}
It suffices to show that the set of places of $K$ such that $G$ is inner has positive natural density in the set of places of $K$.
Indeed, almost all places of $K$ are finite, unramified, and have the property that $G_v$ is quasi-split over $K_v$.
If $G$ is inner to begin with then the conclusion is immediate, as in this setting split and quasi-split coincide.
Hence assume that $G$ is outer.
It is well known that $G$ becomes inner over some finite extension $L$ of $K$.
Consequently if $L$ embeds into $K_v$, then $G_v$ is inner.
One way this can happen is if $L_w=K_v$ for some place $w$ of $L$ lying over $v$, as then $L\subset L_w=K_v$.
Hence the set of all places such that $G_v$ is inner includes the set of all $v$ which split completely in $L$.
An application of the Chebotarev density theorem then gives that $S$ has positive natural density.
\end{proof}

\begin{cor}\label{PrimeDenseCor}
Let $A$ be the subset of primes in the natural numbers such that $p=q_v$ for some $v\in S$, then $A$ has positive natural density in the set of rational primes.
\end{cor}

\begin{proof}[Proof of Theorem \ref{ProgTheorem}]
Let $A$ be as in Corollary \ref{PrimeDenseCor} and for each prime $p\in A$ choose some $v_p\in S$ such that $p=q_v$.
The Green--Tao theorem implies that $A$ has arbitrarily long arithmetic progressions, namely for each natural number $k$ we may find natural numbers $a,b$ such that $\{p_i\}_{i=1}^k=\{ai+b\}_{i=1}^k\subset A$.
If $\gamma_{abs}$ is an absolutely primitive element with length $\ell=\leng(c_{\gamma_{abs}})$, then an inductive application of Proposition \ref{progglueinglemma} to the $v_{p_i}$ yields a $C$ such that
$$\{C(ai+b)\ell\}_{i=1}^k\subset\{Cp_1^{t_1}\cdot~\dots~\cdot p_k^{t_k}\ell\}_{(t_1,\dots, t_k)\in\N^k}\subset\mathcal{L}_p(M).$$
This gives the requisite arithmetic progression.
\end{proof}


\section{The proofs of Theorem \ref{EveryPrimTheorem} and Theorem \ref{CommInvTheorem}.}\label{OtherThmSection}

\noindent We now move on to completing the proofs of Theorems \ref{ProgTheorem}, \ref{EveryPrimTheorem}, and \ref{CommInvTheorem}.
As with Theorem \ref{ProgTheorem}, we first exhibit Theorem \ref{EveryPrimTheorem} when $\Gamma=G(\mathcal{O}_K)$.
This will follow from a minor modification on our choice of admissible $\eta$ in the proof of Propositions \ref{primepowerlengprop} and \ref{progglueinglemma}.
We then prove Theorem \ref{CommInvTheorem}, which will follow from a combinatorial argument using Van der Waerden's theorem.
Having proved Theorems \ref{ProgTheorem} and \ref{EveryPrimTheorem} in the case of $\Gamma=G(\mathcal{O}_K)$, Theorem \ref{CommInvTheorem} in conjunction with the definition of arithmeticity will yield the general case of Theorems \ref{ProgTheorem} and \ref{EveryPrimTheorem}.

\subsection{Every primitive length occurs in an arithmetic progression for $\Gamma=G(\mathcal{O}_K)$.}

We now prove Theorem \ref{EveryPrimTheorem} in the case of $\Gamma=G(\mathcal{O}_k)$.
The proof of Theorem \ref{ProgTheorem} in the case of $\Gamma=G(\mathcal{O}_k)$ gives that every absolutely primitive geodesic occurs in an arbitrarily long arithmetic progression.
Consequently our strategy is to parlay this to all primitive geodesics.
Indeed by Lemma \ref{absprimlemma} we know that the length of every primitive geodesic is a rational multiple of an absolutely primitive one, so it suffices to show that we can find progressions which clear the denominator of this rational number.

\begin{proof}[Proof of Theorem \ref{EveryPrimTheorem}]
Fix $\ell=\leng(c_\gamma)$ to be the length of some primitive geodesic $c_\gamma$, where $\gamma\in\Gamma^{hyp}$.
Then Lemma \ref{absprimlemma} supplies an absolutely primitive element $\gamma_{abs}$ such that $\ell_{abs}=\ell/j$ for some $j\in\N$, where $\ell_{abs}=\leng(c_{\gamma_{abs}})$.
By the proof of Theorem \ref{ProgTheorem}, we know that for any $k$ there exists a $k$-term arithmetic progression in the primitive length spectrum for $\ell_{abs}$.
In specific, we know from Proposition \ref{progglueinglemma} that there exists a $C$ and a subset of the natural numbers $\{ai+b\}_{i=1}^{k}$ such that $\{C(ai+b)\ell_{abs}\}_{i=1}^{k}\subset\mathcal{L}_p(M)$.

\begin{claim}\label{denomclearclaim}
For any natural number $j$, we may always choose a progression so that $j$ divides $C$.
\end{claim}

\noindent Momentarily assuming Claim \ref{denomclearclaim}, then $C=jC'$ for some natural number $C'$ and hence
$$\{C(ai+b)\ell_{abs}\}_{i=1}^{k}=\{C'(ai+b)\ell\}_{i=1}^{k}\subset\mathcal{L}_p(M),$$
as required.
\end{proof}

\begin{proof}[Proof of Claim \ref{denomclearclaim}]
Recall that in the inductive application of Proposition \ref{progglueinglemma}, the admissible $\eta$ are of the form $\eta_{v_1}^{\alpha_1}\cdots\eta_{v_s}^{\alpha_s}$ and produce the following subsets of the primitive length spectrum
$$\{C'p_1^{t_1}\cdots p_s^{t_s}\ell_{abs}\}_{(t_1,\dots,t_s)\in\N^s}\subset\mathcal{L}_p(M).$$
In this notation,
$$C'=\lcm\{n(\gamma,\eta_{v_1}^{R'}),\dots, n(\gamma,\eta_{v_s}^{R'})\},$$
for some natural number $R'$.
Now write $u^{z_1}_1 u^{z_2}_2\cdots u^{z_k}_k$ for the prime decomposition of $j$.
We claim that there exists fixed $r_1,\dots,r_k\in \N$ and $\zeta_{w_i}\in\Comm(\Gamma)$ such that the elements of the form $g_{\alpha_1,\dots,\alpha_s}=\zeta_{w_1}^{r_1}\cdot~\dots~\cdot\zeta_{w_k}^{r_k}\eta_{v_1}^{\alpha_1}\cdot~\dots~\cdot\eta_{v_s}^{\alpha_s}$ are admissible and produce subsets of the primitive length spectrum
$$\{Cp_1^{t_1}\cdots p_s^{t_s}\ell_{abs}\}_{(t_1,\dots,t_s)\in\N^s}\subset\mathcal{L}_p(M),$$
such that $j$ divides $C$.

Indeed, let $L$ be an extension of $K$ such that $G$ is quasi-split over $L$ and choose any finite place $w_i$ lying over $u_i$.
Further let $\zeta_{w_i}$ denote admissible elements built identically to those in the proof of Lemma \ref{unboundedlemma}.
Applying the arguments of Sections \ref{buildinglemmasection} and \ref{controlsection} we see that there is some $R''$, such that for all $i$ and all $r\ge R''$, $n(\widehat\gamma,\zeta_{w_i}^{r+1})$ is equal to the product of $n(\widehat\gamma,\zeta_{w_i}^{r})$ and some integral power $y_{r,i}\ge 1$ of $q_{w_i}$.
Indeed one can see from Scholium \ref{exponentcorollary} applied to a hyperspecial point in $\mathcal{B}_{w_i}$ that this must be the case.
Now for any $r_i$ such that $r_i\ge R=\max\{R',R''\}$, then applying the arguments of Proposition \ref{progglueinglemma} to the elements $\zeta_{w_1}^{r_1}\cdot~\dots~\cdot\zeta_{w_k}^{r_k}\eta_{v_1}^{\alpha_1}\cdot~\dots~\cdot\eta_{v_s}^{\alpha_s}$ constructs the subsets
$$\{C''q_{w_1}^{y_1}\cdots q_{w_k}^{y_k}p_1^{t_1}\cdots p_s^{t_s}\ell\}_{(t_1,\dots,t_s)\in\N^{s}}\subset\mathcal{L}_p(M),$$
in the primitive length spectrum, where $y_i=\sum_{j=R}^{r_i-1} y_{j,i}$.
By construction
$$C''=n(\widehat\gamma,(\zeta_{w_1}\cdots\zeta_{w_k}\eta_{v_1}\cdots\eta_{v_s})^{R})=\lcm\{n(\widehat\gamma,\zeta_{w_1}^{R}),\dots,n(\widehat\gamma,\zeta_{w_k}^{R}),n(\widehat\gamma,\eta_{1}^{R}),\dots,n(\widehat\gamma,\eta_{s}^{R})\}.$$
Now fix $r_i$ such that $r_i\ge R+z_i$, then $u_i^{z_i}\mid q_{w_i}^{y_i}$ and that for the $g_{\alpha_1,\dots,\alpha_s}=\zeta_{w_1}^{r_1}\cdot~\dots~\cdot\zeta_{w_k}^{r_k}\eta_{v_1}^{\alpha_1}\cdot~\dots~\cdot\eta_{v_s}^{\alpha_s}$, the associated constant is given by
$$C=q_{w_1}^{y_1}q^{y_2}_{w_2}\cdots q_{w_k}^{y_k} C''.$$
Hence using these elements, we know
$$\{Cp_1^{t_1}\cdots p_s^{t_s}\ell\}_{(t_1,\dots,t_s)\in\N^{s}}\subset\mathcal{L}_p(M),$$
where $j\mid C$, as required.
\end{proof}


\subsection{Arithmetic progressions are a commensurability invariant.}

\noindent We now move on to the proof Theorem \ref{CommInvTheorem}, which we briefly outline.
Let $M$ be the locally symmetric space obtained via quotienting by the lattice $\Gamma=G(\mathcal{O}_K)$ and let $M'$ be any commensurable manifold, with common finite sheeted cover $X$.
Given a primitive geodesic on $M'$, we transfer it via $X$ to (a finite cover of) a primitive geodesic on $M$.
Using this geodesic, we create a very long arithmetic progression in $\mathcal{L}_p(M)$ and again, via $X$, transfer the geodesics in this progression back to (finite covers of) primitive geodesics on $M'$.
So long as the initial progression on $M$ is sufficiently long, we will be able to guarantee that we can find a $k$-term progression in this subset of $\mathcal{L}_p(M')$.
We give this rigorously in what follows.

\begin{proof}[Proof of Theorem \ref{CommInvTheorem}]
First some notation.
With $M$, $M'$, and $X$ as above, let degree $d_M$ and $d_{M'}$ denote the degree of the cover $X$ of $M$ and $M'$ respectively.
Fix any $\ell'\in\mathcal{L}_p(M')$ and let $c'_{\gamma'}$ be any fixed primitive geodesic with length $\ell'$.
Let $D=\prod_{d\mid d_{M}}d\prod_{d'\mid d_{M'}}d'$ and fix a natural number $N$ bigger than or equal to the Van der Waerden number $W(\sigma(d_{M})\sigma(d_{M'}),k)$, where $\sigma(m)$ is the function that counts the number of positive divisors of $m$ (including $1$),
We now construct a $k$-term progression in $\mathcal{L}_p(M')$ containing $\ell'$.

Lifting $c'_{\gamma'}$ to $X$ and pushing it down to $M$, we obtain (a finite cover of) a primitive geodesic $c_\gamma$ with length $\ell=d'\ell'/d$ for some $d\mid d_{M}$ and some $d'\mid d_{M'}$.
With $X$ as above, Proposition \ref{progglueinglemma} combined with Claim \ref{denomclearclaim} applied to $D$, furnishes a natural number $C$ such that $C=C'D$ and a subset
$$\{C(ai+b)\ell\}_{i=1}^{N}=\{C'D(ai+b)\ell\}_{i=1}^{N},$$
of the primitive length spectrum of $M$.
Now fix primitive geodesic representatives $\{c_{\gamma_i}\}_{i=1}^N$ for each length in this subset.
For each $1\le i\le N$, we lift $c_{\gamma_i}$ to $X$ and subsequently project them back to $M'$ to obtain (finite covers of) primitive geodesics $\{c'_{\gamma'_i}\}_{i=1}^N$.
This will have the effect of multiplying each length $C(ai+b)\ell$ by some rational number $p/q$ where $p\mid d_M$ and $q\mid d_{M'}$.
There are at most $\sigma(d_M)\sigma(d_{M'})$ many choice of $p/q$.
By assigning a distinct color to each $p/q$, color the set $\{1,\dots, N\}$.
Van der Waerden's theorem then implies that there is a monochromatic $k$-term arithmetic progression $\{a'i+b'\}_{i=1}^k$ in $\{1,\dots, N\}$.
Consequently, we see that
$$\left\{C(a(a'i+b)+b)\ell\right\}_{i=1}^{k}=\left\{\frac{C'D d'p}{qd}\left(a(a'i+b')+b\right)\ell'\right\}_{i=1}^{k}\subset\mathcal{L}_p(M'),$$
for a fixed $p\mid d_M$ and a fixed $q\mid d_{M'}$.
As $d$ divides $d_M$, we see by definition of $D$ that $(C'D d'p/qd)$ is a natural number.
Hence, we have constructed a $k$-term arithmetic progression containing integer multiples of $\ell'$ in $\mathcal{L}_p(M')$, as required.
\end{proof}


\end{document}